\documentclass[12pt]{amsart}

\usepackage[
  margin=30mm,
  marginparwidth=25mm,     % + <- Width of your marginpar
  marginparsep=2mm,       % + <- Gap between text block and marginpar
  bottom=25mm,
  ]{geometry}

\usepackage{amsmath,mathabx}
\usepackage{amssymb}
\usepackage{amsthm}
\usepackage[final]{hyperref}
\usepackage{graphicx}

\newtheorem{thm}{Theorem}[section]
\newtheorem*{PWtheorem}{Paley-Wiener Theorem}
\newtheorem{lem}[thm]{Lemma}
\newtheorem{cor}[thm]{Corollary}

\theoremstyle{definition}

\theoremstyle{definition}
\newtheorem{rem}[thm]{Remark}

\newcommand{\N}{\mathbb{N}}
\newcommand{\R}{\mathbb{R}^d}

\newcommand{\LL}{\mathcal{L}}

\usepackage[pdftex,usenames,dvipsnames]{xcolor}

\title[GWP for NLW in analytic Gevrey spaces]
{Global well-posedness for the nonlinear wave equation in analytic Gevrey spaces}

\author{Daniel Oliveira da Silva}
\author{Alejandro J. Castro}

\address{\newline
       Daniel Oliveira da Silva, Alejandro J. Castro \newline
       Department of  Mathematics, Nazarbayev University, \newline
		010000 Nur-Sultan, Kazakhstan}
\email{daniel.dasilva@nu.edu.kz, alejandro.castilla@nu.edu.kz}

\keywords{Wave equations; well-posedness; analytic; Gevrey spaces}
\subjclass[2010]{35Q40, 35L70}

 \thanks{
A. J. Castro is supported by the Nazarbayev University Faculty Development Competitive Research Grants Program, grant number 110119FD4544.}

\begin{document}

\begin{abstract}
We obtain an asymptotic rate of decay for the radius of spatial analyticity of solutions to the nonlinear wave equation with initial data in the analytic Gevrey spaces.
\end{abstract}

\maketitle

%%%%%%%%%%%%%%%%%%%%%%%%%%%%%%%%%%%%%%%%%%%%%%%%%%%%%%%%%%%%%%%%
\section{Introduction}
%%%%%%%%%%%%%%%%%%%%%%%%%%%%%%%%%%%%%%%%%%%%%%%%%%%%%%%%%%%%%%%%
The nonlinear wave equation (NLW) is the equation
\begin{equation}\label{nlw}
u_{tt} - \Delta u = \mu |u|^{p-1}u.
\end{equation}
Here, $u : \mathbb{R}^{d+1} \rightarrow \mathbb{C}$, $\mu \in \{-1,1\}$, and the exponent $p$ satisfies $1 < p < \infty$.  The case $\mu = -1$ is known as the \emph{defocusing} case, while $\mu = 1$ is called the \emph{focusing} case.  This equation has a long history and many results are known; see \cite{T2006} for a detailed exposition and references therein for details.  With regards to the question of local well-posedness in the homogeneous Sobolev spaces $\dot{H}^{s}$, H. Lindblad and C. D. Sogge \cite{LS1995} showed that the optimal regularity in dimension $d = 2$ is given by
$$
s(p) := 
\left\{
\begin{array}{ll}
    \dfrac{3}{4} - \dfrac{1}{p-1}, & 3 \leq p \leq 5, \\
    1 - \dfrac{2}{p-1}, & p \geq 5. 
\end{array}
\right.$$
See also the refinements of M. Nakamura and T. Ozawa \cite{NaOz1999}.  We remark that these results are stated for the homogeneous Sobolev spaces $\dot{H}^{s}$, but can be extended to the inhomogeneous Sobolev spaces $H^{s}$ by a simple integration in time argument.\\

In the present work, we will consider the Cauchy problem for \eqref{nlw} with initial which belong to the Gevrey class $G^{\sigma, s}(\R)$.  These spaces first appeared in the work of Foias and Temam \cite{FT1989} on the Navier-Stokes equation and are defined as the space of functions for which the norm
\[
\| f \|_{G^{\sigma, s}(\R)} 
:= \left( \int_{\R} e^{2 \sigma |\xi|} \langle \xi \rangle^{2s} | \widehat{f}(\xi) |^{2}\ d\xi \right)^{1/2}
\]
is finite.  Here, $\widehat{f}$ denotes the Fourier transform
\[
\widehat{f}(\xi) 
:= \mathfrak{F}(f) (\xi)
:= \int_{\R} e^{- i x \cdot \xi} f(x) \, dx,
\]
and $\langle z \rangle := (1 + |z|^2)^{1/2}$.  In the case $s = 0$, we write $G^{\sigma, 0} = G^{\sigma}$.  
Note also that $G^{0, s} = H^{s}$. The interest in these spaces is due to the following theorem:
\begin{PWtheorem}\label{wiener}
Let $\sigma > 0$ and $f \in L^{2}(\mathbb{R})$.  Then the following are equivalent:
\begin{enumerate}
\item $f \in G^{\sigma}(\mathbb{R})$;
\item $f$ is the restriction to $\mathbb{R}$ of a function $F$ which is holomorphic in the strip
\[
S_{\sigma} = \{ x + iy \in \mathbb{C} :\ |y| < \sigma \},
\]
and satisfies
\[
\sup_{|y| < \sigma}\| F( \cdot + i y) \|_{L^{2}_{x}(\mathbb{R})} < \infty.
\]
\end{enumerate}
\end{PWtheorem}
\noindent A proof of this result can be found on \cite[p. 174]{K1976}.  We remark that the implication $(1) \Rightarrow (2)$ also holds in higher dimensions; the proof is a simple modification of the original.  Note that this result also holds with the spaces $L^{2}$ replaced with $H^{s}$ and $G^{\sigma}$ replaced with $G^{\sigma, s}$. \\

In recent years, many authors have considered the Cauchy problem for a variety of equations with initial data in $G^{\sigma, s}$ spaces; see, for example \cite{
BHP2017,
BGK2005, 
BGK2006,
D2017,
D2018,
GHHP2013, 
GK2002,
GK2003,
HHP2011,
HKS2017,
HP2012, 
H2017,
H2020,
HO2014,
L2012,
SD2017, 
ST2015, 
T2017} 
for some of the more recent works on this subject.  It should be noted that all of the works mentioned here are concerned with equations which are first-order in time.  The only result known to the authors involving second-order equations is the result of Y. Guo and E. S. Titi \cite{GT2013}, which deals with general nonlinear wave equations in the periodic setting.  With these facts in mind, we present our main results, which are the content of the following theorems:
\begin{thm}\label{mainthm1}
Let $p > 1$ be an odd integer.  Then the Cauchy problem
\begin{equation}\label{NLWCauchy}
\left\{
\begin{array}{l}
 u_{tt} - \Delta u + |u|^{p-1}u = 0, \\
 u(\cdot,0) = u_0 \in G^{\sigma, s}(\R), \\
 u_{t}(\cdot,0) = u_1 \in G^{\sigma, s-1}(\R),
\end{array}
\right.
\end{equation}
is unconditionally locally well-posed in $G^{\sigma, s}(\R) $ $\times G^{\sigma, s-1}(\R)$, provided that $s > d/2 - 1/p$ and $\sigma > 0$. That is, for each $u_0 \in G^{\sigma, s}(\R)$ and $u_1 \in G^{\sigma, s-1}(\R)$, there exists $\delta > 0$ such that equation \eqref{NLWCauchy} has a unique solution
\[
u \in C([0,\delta); G^{\sigma, s}(\R)) \cap C^{1}([0,\delta); G^{\sigma, s-1}(\R)).
\]
Moreover, the solution depends continuously on the initial data.  Thus, the analyticity of solutions persists for sufficiently small times.
\end{thm}

\begin{thm}\label{mainthm2}
Let $d = 1$ or $2$, $s$ as in Theorem \ref{mainthm1}, and let $p > 1$ be an odd integer.  Suppose $u$ is a smooth local solution to the problem in equation \eqref{NLWCauchy}, and that $u_0 \in G^{\sigma_0, s}(\R)$ and $u_1 \in G^{\sigma_0, s-1}(\R)$.  Then, for any $T>0$,
\begin{equation}\label{eq:unifbound}
\sup_{t \in [0,T]} \| u(\cdot,t) \|_{G^{\sigma(T), s}(\R)} 
+ \sup_{t \in [0,T]} \| u_{t}(\cdot,t) \|_{G^{\sigma(T), s-1}(\R)} < \infty,
\end{equation}
provided that,
\[
\sigma(T) = \min \Big\{ \sigma_0, \frac{C}{(1+T)^{(p+1)/2}} \Big\}
\]
when $d = 1$, and
\[
\sigma(T) = \min \Big\{ \sigma_0, \frac{C}{(1+T)^{(p+1-\varepsilon)/(1-\varepsilon)}} \Big\}
\]
when $d = 2$, for any $\varepsilon > 0$ and some constant $C>0$ which is independent of $T$.  Thus, solutions exist globally, and the radius of analyticity $\sigma$ for $u$ satisfies $\sigma \geq \sigma(T)$.
\end{thm}

Theorem \ref{mainthm1} will be proved in Section \ref{local}, and Theorem \ref{mainthm2} will be proved in Section \ref{global}.  Before we prove these results, we first state some preliminary material in Section \ref{prel}.

\begin{rem}
Theorem \ref{mainthm1} is by no means optimal.  We expect that it is possible to reduce $s$ below $d/2 - 1/p$ in dimensions $d \geq 2$ by applying Strichartz-type estimates, though this would then become a conditional well-posedness result.  Moreover, since our goal is to study the evolution of the analyticity of solutions, we will not pursue this here.
\end{rem}

%--------------------------------
\section{Preliminaries}\label{prel}
%--------------------------------

To set up the proofs of Theorems \ref{mainthm1} and \ref{mainthm2}, let us first fix the notation to be used.  In addition to the Fourier transform $\mathfrak{F}$, let us denote the inverse Fourier transform by
\[
\widecheck{f}(x) 
:=
\mathfrak{F}^{-1}(f)(x)
:= \frac{1}{(2\pi)^d}\int_{\R} e^{i x \cdot \xi} f(\xi) \, d\xi.
\]
For any $s, \sigma \in \mathbb{R}$, we define the pseudodifferential operators $e^{\sigma|D|}\langle D \rangle^{s}$ and $|\nabla|^{s}$ by the Fourier multipliers
\[
e^{\sigma|D|}\langle D \rangle^{s} f
:= \mathfrak{F}^{-1} \left( e^{\sigma|\xi|} \langle \xi \rangle^{s} \widehat{f}(\xi) \right)
\]
and
\[
|\nabla|^{s} f
:= \mathfrak{F}^{-1} \left(|\xi|^{s} \widehat{f}(\xi) \right).
\]
We will denote constants which can be determined by known parameters in a given situation by $C$, but whose values are not crucial to the problem at hand and may differ from line to line.  We also write $a\lesssim b$ as shorthand for $a\leq Cb$ and $a\sim b$ when $a\lesssim b$ and $b\lesssim a$.\\

Next, we state some estimates which will be useful throughout.  The first is the following embedding lemma:
\begin{lem}\label{embed}
Let $s, s' \in \mathbb{R}$, and $0 \leq \sigma' < \sigma$. 
Then
\[
\| f \|_{G^{\sigma', s'}(\R)} 
\lesssim \| f \|_{G^{\sigma,s}(\R)},
\]
and hence $G^{\sigma,s}(\R) \hookrightarrow G^{\sigma',s'}(\R)$.
\end{lem}
\begin{proof}
It is well-known that
\[
\langle \xi \rangle^{s'-s} \lesssim e^{(\sigma - \sigma')|\xi|}
\]
for $\sigma > \sigma'$.  Multiplying by 
$\langle \xi \rangle^{s} e^{\sigma'|\xi|}$ we obtain
\[
\langle \xi \rangle^{s'}e^{\sigma'|\xi|} 
\lesssim  \langle \xi \rangle^{s} e^{\sigma|\xi|},
\]
and the desired result immediately follows.
\end{proof}

Another very useful inequality is the following generalized Sobolev product estimate  from \cite{DFS2010}.
\begin{lem}\label{Lem:prodSobolev}
Let $d \geq 1$, $f \in H^{s_1}(\R)$ and  $g \in H^{s_2}(\R)$. Then, 
$fg \in H^{-s_0}(\R)$ and
\begin{equation}\label{eq:productSobolev}
  \|fg\|_{H^{-s_0}(\R)}
  \lesssim \|f\|_{H^{s_1}(\R)} \|g\|_{H^{s_2}(\R)},
\end{equation}
provided that 
$$s_0+s_1+s_2 \geq \max\{s_0,s_1,s_2\} 
\qquad \text{and} \qquad
s_0+s_1 + s_2 \geq \frac{d}{2},$$
but the equality cannot hold in both relations at the same time.
\end{lem}
\noindent As a consequence of Lemma \ref{Lem:prodSobolev} we obtain:
\begin{lem}\label{Lem:ppowerGevrey}
Let $d \geq 1$, $\sigma \geq 0$, and let $p \geq 3$ be an odd positive integer.  Then, for every $f \in G^{\sigma, s}(\R)$,
\begin{equation}\label{eq:ppowerGevrey}
  \|f^p\|_{G^{\sigma, s-1}(\R)}
  \lesssim \|f\|^p_{G^{\sigma, s}(\R)}.
\end{equation}
whenever
\[
s \geq \max \left\{ \frac{d}{2} - \frac{1}{p},\ \frac{1}{2} \left( 1 - \frac{1}{p} \right)  \right\}.
\]
\end{lem}

\begin{proof}
By definition of the Gevrey spaces we have that
\begin{equation}\label{eq:def}
\|f^p\|_{G^{\sigma, s-1}(\R)}    
 = \| \langle \xi \rangle^{s-1} e^{\sigma |\xi|} \widehat{f^p}(\xi) \|_{L^2_\xi(\R)}.
\end{equation}
On the other hand,
\begin{align*}
|e^{\sigma |\xi|} \widehat{f^2}(\xi) |
& = |e^{\sigma |\xi|} \widehat{f}*\widehat{f}(\xi)|
\leq \int_{\R} e^{\sigma |\eta|} |\widehat{f}(\eta)|
e^{\sigma |\xi - \eta|} |\widehat{f}(\xi - \eta)| \, d\eta \\
& = \Big(e^{\sigma |\cdot |} |\widehat{f}| * e^{\sigma |\cdot |} |\widehat{f}|\Big) (\xi) 
 = \Big(
\mathfrak{F}(e^{\sigma |D |} \mathfrak{F}^{-1}(|\widehat{f}|))
* 
\mathfrak{F}(e^{\sigma |D |} \mathfrak{F}^{-1}(|\widehat{f}|))
\Big) (\xi) \\
& = \mathfrak{F}\Big( (e^{\sigma |D |} \mathfrak{F}^{-1}(|\widehat{f}|))^2 \Big)(\xi).
\end{align*}
Iterating this relation, we can control \eqref{eq:def} by
\begin{equation*}
\|f^p\|_{G^{\sigma, s-1}(\R)}  
\leq \|F^p\|_{H^{s-1}(\R)}, 
\end{equation*}
where
$$F(x):=e^{\sigma |D |} \mathfrak{F}^{-1}(|\widehat{f}|)(x).$$
Then, in order to justify the estimate \eqref{eq:ppowerGevrey} 
it is enough to see that
\begin{equation}\label{eq:newgoal}
\|F^p\|_{H^{s-1}(\R)}
\lesssim \|F\|_{H^{s}(\R)}^p.
\end{equation}
As for \eqref{eq:newgoal}, it suffices to check that
\begin{equation}\label{eq:goal}
\| F^{p-k} \|_{H^{s-1+\frac{k}{p}} (\R)} 
\lesssim \| F^{p-k-1} \|_{H^{s-1+\frac{k+1}{p}} (\R)} 
\| F \|_{H^{s}(\R)}, 
\quad k = 0,\ldots,p-1,
\end{equation}
but this follows from the product estimate in Lemma \ref{Lem:prodSobolev} for $s$ as in the statement of the lemma.
\end{proof}

We conclude this section with a lemma which will be very helpful in Section \ref{global} below.

\begin{lem}\label{op}
Let $\sigma \geq 0$, $p \geq 3$ be an odd number and let $\LL$ be the operator given by
\[
\LL f
:= |e^{\sigma|D|}f|^{p-1} e^{\sigma|D|}f 
- e^{\sigma|D|}(|f|^{p-1} f).
\]
Then for $d = 1$, we have the estimate
\[
\| \LL f \|_{L^{2}(\mathbb{R})} 
\lesssim \sigma^{\theta}  \, 
\| f \|_{G^{\sigma, 0}(\mathbb{R})}^{\frac{p+1-2\theta}{2}} \, 
\| \nabla f \|_{G^{\sigma, 0}(\mathbb{R})}^{\frac{p-1+2\theta}{2}},
\quad 0 \leq \theta \leq 1,
\]
while for $d = 2$,
\[
\|  \LL f \|_{L^{2}(\mathbb{R}^{2})} 
\lesssim \sigma^{\theta} \, 
\| f \|_{G^{\sigma,0}(\mathbb{R}^{2})}^{p-1+\theta} \, 
\| \nabla f \|_{G^{\sigma,0}(\mathbb{R}^{2})}^{1-\theta},
\quad 0 < \theta < 1.
\]
\end{lem}

\begin{proof}
Let us introduce the notation
$$F(x) 
:= e^{\sigma|D|}f(x)$$
and
$$G(x) 
:= \mathfrak{F}^{-1}(|\widehat{F}|)(x).$$
Following the proof of Lemma 7 in \cite{T2017}, we have that
\[\begin{aligned}
\widehat{\LL f }(\xi) 
& = \int_{H} 
\Big( 1 - e^{- \sigma (\sum_{k = 1}^{p}|\eta_k| - |\xi| )} \Big) \widehat{F}(\eta_1) \widehat{\overline{F}}(\eta_2)\ \times \\
& \qquad \times \widehat{F}(\eta_3) \widehat{\overline{F}}(\eta_4) \cdots \widehat{F}(\eta_{p-2})\widehat{\overline{F}}(\eta_{p-1}) \widehat{F}(\eta_{p})\ d\eta_1 \cdots d\eta_{p-1},
\end{aligned}\]
where $H$ is the hyperplane $\xi = \eta_1 + \cdots + \eta_{p}$.  Next, we observe that
\[
1 - e^{-x} \leq 1 \qquad \textrm{and} \qquad 1 - e^{-x} \leq x.
\]
Interpolating, we see that
\[
1 - e^{-x} \leq x^{\theta}, \quad
0 \leq \theta \leq 1.
\]
Hence,
\[\begin{aligned}
1 - e^{- \sigma(\sum_{k = 1}^{p}|\eta_k| - |\xi| )} 
& \leq \sigma^{\theta}\Big(\sum_{k = 1}^{p}|\eta_k| - |\xi| \Big)^{\theta} 
\leq
\sigma^{\theta} \sum_{k = 1}^{p} |\eta_k|^{\theta}.
\end{aligned}\]
Plancherel's identity then yields
\begin{equation}\label{operatorest}
\| \LL f \|_{L^{2}(\R)} 
%= 
%\| \widehat{\LL f} \|_{L^{2}(\R)} 
\lesssim \sigma^{\theta} \, 
\| (|\nabla|^{\theta} G) \, G^{p-1} \|_{L^{2}(\R)}.
\end{equation}

\quad

To proceed, we consider each dimension separately.  For $d = 1$, we estimate the term on the right hand side of equation \eqref{operatorest} by
\begin{equation}\label{split}\begin{aligned}
%\| \widehat{\LL f} \|_{L^{2}(\mathbb{R})} 
%& \lesssim \sigma^{\theta} \, 
\| (|\nabla|^{\theta} G) \, G^{p-1} \|_{L^{2}(\mathbb{R})} 
 \lesssim %\sigma^{\theta} \,  
\| |\nabla|^{\theta} G \|_{L^{2}(\mathbb{R})} \,
\| G \|_{L^{\infty}(\mathbb{R})}^{p-1}.
\end{aligned}\end{equation}
For the first norm above, we apply again Plancherel to see that
\[
\| |\nabla|^{\theta} G \|_{L^{2}(\mathbb{R})} 
= \| |\xi|^{\theta} \widehat{G} \|_{L^{2}(\mathbb{R})} 
= \| |\xi|^{\theta} |\widehat{F}| \|_{L^{2}(\mathbb{R})} 
= \| F \|_{\dot{H}^{\theta}(\mathbb{R})}.
\]
Interpolating, we have that
\[
\| F \|_{\dot{H}^{\theta}(\mathbb{R})} 
\leq
\| F \|_{L^{2}(\mathbb{R})}^{1-\theta} \,
\| \nabla F \|_{L^{2}(\mathbb{R})}^{\theta},
\]
for $0 \leq \theta \leq 1$.  For the second term on the right hand side of equation \eqref{split}, we apply the Gagliardo-Nirenberg inequality to obtain
\[
\| G \|_{L^{\infty}(\mathbb{R})} 
\lesssim \| G \|_{L^{2}(\mathbb{R})}^{1/2} \,
\| G \|_{\dot{H}^{1}(\mathbb{R})}^{1/2} 
= \| F \|_{L^{2}(\mathbb{R})}^{1/2} \,
\| \nabla F \|_{L^{2}(\mathbb{R})}^{1/2}.
\]
Combining previous estimates, we see that
\[
\| \LL f \|_{L^{2}(\mathbb{R})} 
\lesssim \sigma^{\theta} \, 
\| F \|_{L^{2}(\mathbb{R})}^{\frac{p+1-2\theta}{2}} \,
\| \nabla F \|_{L^{2}(\mathbb{R})}^{\frac{p-1+2\theta}{2}},
\]
which leads to the desired result.\\

For the case of $d=2$, we cannot use the Gagliardo-Nirenberg inequality as before.  Instead, we begin by applying H\"older's inequality 
%to the first term on the right side of equation \eqref{split} to obtain
in \eqref{operatorest} to obtain
\begin{equation}\label{holder}
\| (|\nabla|^{\theta} G ) \, G^{p-1} \|_{L^{2}(\mathbb{R}^2)}
\leq
\||\nabla|^{\theta} G \|_{L^{2/\theta}(\mathbb{R}^2)} \, 
\| G \|_{L^{\frac{2(p-1)}{1-\theta}}(\mathbb{R}^2)}^{p-1}. 
\end{equation}
By Sobolev embedding, we have that
\[
\| |\nabla|^{\theta} G \|_{L^{2/\theta}(\mathbb{R}^2)} 
\lesssim \| |\nabla|^{\theta} G \|_{\dot{H}^{1-\theta}(\mathbb{R}^{2})} 
\lesssim \| \nabla f \|_{G^{\sigma,0}(\mathbb{R}^{2})}
\]
and
\[
\| G \|_{L^{\frac{2(p-1)}{1-\theta}}(\mathbb{R}^2)} 
\lesssim \| G \|_{\dot{H}^{\alpha}(\mathbb{R}^{2})}
\]
with $\alpha$ given by
\[
\alpha := \frac{p-2+\theta}{p-1}.
\]
Observe that $0 < \alpha < 1$.  Thus, we can interpolate again to obtain
\[
\| G \|_{\dot{H}^{\alpha}(\mathbb{R}^{2})}
\leq
\| G \|_{L^{2}(\mathbb{R}^{2})}^{1-\alpha} \,
\| \nabla G \|_{L^2(\mathbb{R}^{2})}^{\alpha}
= \| f \|_{G^{\sigma,0}(\mathbb{R}^{2})}^{1-\alpha} \
\| \nabla f \|_{G^{\sigma,0}(\mathbb{R}^{2})}^{\alpha}.
\]
Substituting the above estimates into equation \eqref{holder}, we see that
\[
%\| (|\nabla|^{\theta} G) \, G^{p-1} \|_{L^{2}(\mathbb{R}^2)} 
\| \LL f \|_{L^{2}(\mathbb{R}^2)} 
\lesssim 
\sigma^\theta \| \nabla f \|_{G^{\sigma,0}(\mathbb{R}^2)}^{p-1+\theta} 
\| f \|_{G^{\sigma,0}(\mathbb{R}^2)}^{1-\theta},
\]
as claimed.
\end{proof}

%%%%%%%%%%%%%%%%%%%%%%%%%%%%%%%%%%%%%%%%%%%%%%%%%%%%%%%%%%%%%%%%
\section{Proof of Theorem \ref{mainthm1}}\label{local}
%%%%%%%%%%%%%%%%%%%%%%%%%%%%%%%%%%%%%%%%%%%%%%%%%%%%%%%%%%%%%%%%

In this section we prove Theorem \ref{mainthm1}.  We proceed via a standard fixed-point argument, where the iteration takes place in the space
\[
C([0,\delta); G^{\sigma,s}(\R)) \cap C^{1}([0,\delta); G^{\sigma,s-1}(\R)).
\]
We start recalling the following result for inhomogeneous linear wave equations in Sobolev spaces
(\cite[p. 79]{T2006}).
\begin{lem}\label{linearwave}
Let $s \in \mathbb{R}$, $d \geq 1$, and assume that
\[
F \in L^{1}([0,T]; H^{s-1}(\R)).
\]
Suppose $u$ is a solution to the Cauchy problem
\begin{equation}\label{Cauchy}
\left\{
\begin{array}{l}
 u_{tt} - \Delta u = F(x,t), \\
 u(\cdot,0) = u_0 \in H^{s}(\R) , \\
 u_{t}(\cdot,0) = u_1 \in H^{s-1}(\R).
\end{array}
\right.
\end{equation}
Then $u$ satisfies the energy estimate
\[\begin{aligned}
& \sup_{t \in [0,T]}\| u(\cdot,t) \|_{H^{s}(\R)} 
+ \sup_{t \in [0,T]} \| u_{t}(\cdot,t) \|_{H^{s-1}(\R)} \\
& \qquad \quad \lesssim \langle T \rangle 
\left(\| u_0 \|_{H^{s}(\R)} 
+ \| u_{1} \|_{H^{s-1}(\R)}
+ \int_{0}^{T} \| F(\cdot,\tau) \|_{H^{s-1}(\R)}\ d \tau \right).
\end{aligned}\]
\end{lem}
By applying the pseudodifferential operator $e^{\sigma|D|}$ to all expressions in equation \eqref{Cauchy}, we obtain the following corollary:
\begin{cor}\label{linearwavegevrey}
Let $s \in \mathbb{R}$, $\sigma \geq 0$, $d \geq 1$, and assume that
\[
F \in L^{1}([0,T]; G^{\sigma, s-1}(\R)).
\]
Suppose $u$ is a solution to the Cauchy problem \eqref{Cauchy} with $u_{0} \in G^{\sigma, s}(\R)$ and $u_{1} \in G^{\sigma, s-1}(\R)$.  Then $u$ satisfies the modified energy estimate
\begin{equation}\label{energy}\begin{aligned}
& \sup_{t \in [0,T]}\| u(\cdot,t) \|_{G^{\sigma, s}(\R)} 
+ \sup_{t \in [0,T]} \| u_{t}(\cdot,t) \|_{G^{\sigma, s-1}(\R)} \\
& \qquad \quad \lesssim \langle T \rangle 
\left(\| u_0 \|_{G^{\sigma, s}(\R)} 
+ \| u_{1} \|_{G^{\sigma, s-1}(\R)}
+ \int_{0}^{T} \| F(\cdot,\tau) \|_{G^{\sigma, s-1}(\R)}\ d \tau \right).
\end{aligned}\end{equation}
\end{cor}

Next, we recall that solutions to equation \eqref{Cauchy} can be written in the Duhamel form
\[
u(x,t) = W'(t) * u_{0}(x) + W(t) * u_{1}(x) + \int_{0}^{t} W(t-s) * F(s,x)\ ds,
\]
where $W(t)$ is the operator with symbol
\[
\widehat{W}(t) = \frac{\sin (t |\xi|)}{|\xi|}.
\]
For fixed $(u_0, u_1) \in G^{\sigma, s}(\R) \times G^{\sigma, s-1}(\R)$, define the mapping $\Phi$ on the space \[C([0,\delta); G^{\sigma, s}(\R)) \cap C^{1}([0,\delta);G^{\sigma, s-1}(\R))\] by
\[
\Phi(u)(x,t) 
:= W'(t) * u_{0}(x) 
+ W(t) * u_{1}(x) 
+ \int_{0}^{t} W(t-s) * \Big( -|u(x,s)|^{p-1} u(x,s) \Big) \ ds.
\]
To prove Theorem \ref{mainthm1}, it suffices to show that $\Phi$ has a fixed point.  For this, observe that $\Phi(u)$ is a solution to the problem
\[
\left\{ \begin{array}{l}
  (\partial_t^2 - \Delta)\Phi(u)+ |u|^{p-1} u = 0, \\
 \Phi(u)(\cdot,0) = u_{0}, \\
 \Phi(u)_t(\cdot,0) = u_{1}.
\end{array}
\right.\]
Moreover, for any $u,v$ in a ball of radius $R$ centered at 0 in the solution space, the difference $\Phi(u) - \Phi(v)$ satisfies
\[
\left\{ \begin{array}{l}
 (\partial_t^2 - \Delta) (\Phi(u) - \Phi(v)) = |v|^{p-1}v - |u|^{p-1} u, \\
 (\Phi(u) - \Phi(v))(x,0) = 0, \\
 (\Phi(u) - \Phi(v))_{t}(x,0) = 0.
\end{array}
\right.\]
Applying the Gevrey energy estimate in equation \eqref{energy}, we see that
\[\begin{aligned}
& \sup_{t \in [0,\delta)}
\| \Phi(u)(\cdot,t) - \Phi(v)(\cdot,t) \|_{G^{\sigma, s}(\R)} 
+ \sup_{t \in [0,\delta)} 
\| \Phi(u)_{t}(\cdot,t) - \Phi(v)_{t}(\cdot,t) \|_{G^{\sigma, s-1}(\R)} \\
& \qquad \qquad  \lesssim \langle \delta \rangle \int_{0}^{\delta} \left\| |v|^{p-1}v(\cdot,\tau) - |u|^{p-1} u(\cdot,\tau) \right\|_{G^{\sigma, s-1}}\ d\tau.
\end{aligned}\]
A simple computation and a modification of the proof of Lemma \ref{Lem:ppowerGevrey} show that
\begin{align*}
& \left\| |v|^{p-1}v(\cdot,\tau) - |u|^{p-1} u(\cdot,\tau) \right\|_{G^{\sigma, s-1}(\R)} \\
& \qquad \qquad \lesssim \Big( \| u(\cdot,\tau) \|_{G^{\sigma, s}(\R)}^{p-1} 
+ \| v(\cdot,\tau) \|_{G^{\sigma, s}(\R)}^{p-1} \Big) \| u(\cdot,\tau) - v(\cdot,\tau) \|_{G^{\sigma, s}(\R)},
\end{align*}
for $s \geq d/2 - 1/p$. Hence, it follows that
\begin{align*}
& \sup_{t \in [0,\delta)} \| \Phi(u)(\cdot,t) - \Phi(v)(\cdot,t) \|_{G^{\sigma, s}(\R)}
+ \sup_{t \in [0,\delta)} \| \Phi(u)_{t}(\cdot,t) - \Phi(v)_{t}(\cdot,t) \|_{G^{\sigma, s-1}(\R)} \\
&  \qquad \lesssim \delta \langle \delta \rangle \Big( \sup_{t \in [0,\delta)}\| u(\cdot,t) \|_{G^{\sigma, s}(\R)}^{p-1} 
+ \sup_{t \in [0,\delta)} \| v(\cdot,t) \|_{ G^{\sigma, s}(\R)}^{p-1} \Big) 
\sup_{t \in [0,\delta)} \| u(\cdot,t) - v(\cdot,t) \|_{ G^{\sigma, s}(\R)} \\
&  \qquad \lesssim \delta \langle \delta \rangle R^{p-1} \sup_{t \in [0,\delta)} \| u(\cdot,t) - v(\cdot,t) \|_{ G^{\sigma, s}(\R)}.
\end{align*}
If $\delta > 0$ is sufficiently small, we deduce that
\begin{align*}
& \sup_{t \in [0,\delta)} \| \Phi(u)(\cdot,t) - \Phi(v)(\cdot,t) \|_{G^{\sigma, s}(\R)}
+ \sup_{t \in [0,\delta)} \| \Phi(u)_{t}(\cdot,t) - \Phi(v)_{t}(\cdot,t) \|_{G^{\sigma, s-1}(\R)} \\
&  \qquad <  \sup_{t \in [0,\delta)} \| u(\cdot,t) - v(\cdot,t) \|_{ G^{\sigma, s-1}(\R)}
+ \sup_{t \in [0,\delta)} \| u_t(\cdot,t) - v_t(\cdot,t) \|_{ G^{\sigma, s-1}(\R)},
\end{align*}
%\[
%\| \Phi(u) - \Phi(v) \|_{L^{\infty}_{t}G^{\sigma, s}_{x}} + \| \Phi(u)_{t} - \Phi(v)_{t} \|_{L^{\infty}_{t}G^{\sigma, s-1}_{x}} < \| u - v \|_{L^{\infty}_{t}G^{\sigma, s}_{x}} + \| u_{t} - v_{t} \|_{L^{\infty}_{t}G^{\sigma, s-1}_{x}},
%\]
so that $\Phi$ is a contraction.  The existence of a unique fixed point follows from the Banach fixed point theorem.\\

By a similar argument, we can shows that solutions depend continuously on the initial data. Thus, the Cauchy problem for the equation \eqref{NLWCauchy} is locally well-posed in 
$G^{\sigma, s}(\R) \times G^{\sigma, s-1}(\R)$. \\

%%%%%%%%%%%%%%%%%%%%%%%%%%%%%%%%%%%%%%%%%%%%%%%%%%%%%%%%%%%%%%%%
\section{Proof of Theorem \ref{mainthm2}}\label{global}
%%%%%%%%%%%%%%%%%%%%%%%%%%%%%%%%%%%%%%%%%%%%%%%%%%%%%%%%%%%%%%%%

In this section, we prove Theorem \ref{mainthm2}.  We consider two cases.

\subsection{Case 1: $s = 1$} To begin, we first observe that
\[
\| u(\cdot,t) \|_{G^{\sigma,1}(\R)} 
\sim \| u(\cdot,t) \|_{G^{\sigma,0}(\R)} 
+ \| \nabla u(\cdot,t) \|_{G^{\sigma,0}(\R)}.
\]
This follows from the analogous result for Sobolev spaces (see, for example, \cite[Appendix A]{T2006}).
To estimate the first term, we apply a simple integration in time argument to show that
\[
\| u(\cdot,t) \|_{G^{\sigma,0}(\R)}
\leq \| u(\cdot,0) \|_{G^{\sigma,0}(\R)} 
+ \int_{0}^{t} \| u_{t} (\cdot,\tau) \|_{G^{\sigma,0}(\R)}\ d \tau.
\]
Next, we define the quantity
\[
E_{\sigma}(t) 
:= \frac{1}{2}\| \nabla u(\cdot,t) \|_{G^{\sigma,0}(\R)}^2 
+ \frac{1}{2} \| u_{t}(\cdot,t) \|_{G^{\sigma,0}(\R)}^2 
+ \frac{1}{p+1}\left\| e^{\sigma|D|} u (\cdot,t) \right\|_{L^{p+1}(\R)}^{p+1}.
\]
We remark that $E_{0}(t)$ is the conserved energy for the equation \eqref{nlw}.  It is then easy to see that
\begin{equation}\label{lowerorder}
\| u(\cdot,t) \|_{G^{\sigma,0}(\R)} 
\lesssim \| u(\cdot,0) \|_{G^{\sigma,0}(\R)} 
+ \int_{0}^{t} E^{1/2}_{\sigma}(\tau)\ d \tau.
\end{equation}
Moreover, since
\[
\| u_t(\cdot,t) \|_{G^{\sigma, 0}(\R)} 
\lesssim E^{1/2}_{\sigma}(t)
\]
and
\[
\| \nabla u(\cdot,t) \|_{G^{\sigma, 0}(\R)} 
\lesssim E^{1/2}_{\sigma}(t),
\]
in order to conclude \eqref{eq:unifbound},
it suffices to show that $E_{\sigma}(t)$ remains bounded in the interval $[0,T]$.\\

To show that this is the case, we will use a bootstrap argument, where the parameter $\sigma$ will play a crucial role in ``closing the bootstrap''.  For any $t \in [0,T]$, let H$(t)$ and C$(t)$ be the statements
\begin{itemize}
\item H$(t)$: $E_{\sigma}(\tau) \leq 4E_{\sigma}(0)$ for $0 \leq \tau \leq t$,
\item C$(t)$: $E_{\sigma}(\tau) \leq 2E_{\sigma}(0)$ for $0 \leq \tau \leq t$.
\end{itemize}
To close the bootstrap, we must prove the following four statements:
\begin{enumerate}
\item[$a)$] H$(t) \Rightarrow$ C$(t)$;
\item[$b)$] C$(t) \Rightarrow$ H$(t')$ for all $t'$ in a neighborhood of $t$;
\item[$c)$] If $\{t_n\}_{n \in \N}$ is a sequence in $[0,T]$ such that $t_n \rightarrow t \in [0,T]$, with C$(t_{n})$ true for all $t_n$, then C$(t)$ is also true;
\item[$d)$] H$(t)$ is true for at least one $t \in [0,T]$.
\end{enumerate}

\begin{proof}[Proof of $a)$]

Fix $t \in [0,T]$ and assume that H$(t)$ holds.
Define
$$U(x,t)
:= e^{\sigma |D|} u(x,t),$$
where $u$ is the local solution of the Cauchy problem \eqref{NLWCauchy}.
It is clear that
\begin{equation}\label{eq:NLWforU}
    U_{tt}-\Delta U = -e^{\sigma|D|}(|u|^{p-1} u).
\end{equation}
Moreover, the modified energy $E_\sigma(t)$ can be written as
\[
E_{\sigma}(t)
= \frac{1}{2} \| \nabla U(\cdot,t) \|_{L^{2}(\R)}^2 
+ \frac{1}{2} \| U_{t}(\cdot,t) \|_{L^{2}(\R)}^2 
+ \frac{1}{p+1}\left\| U(\cdot,t) \right\|_{L^{p+1}(\R)}^{p+1}.
\]
Then, for $\tau \leq t$, we use the Fundamental Theorem of Calculus, an integration by parts, equation \eqref{eq:NLWforU}  and  the Cauchy-Schwarz inequality to deduce
\begin{equation}\label{intenergy}
\begin{aligned}
E_{\sigma}(\tau) 
& = E_{\sigma}(0) + \int_{0}^{\tau} \frac{d E_{\sigma}(\tau')}{d\tau}\ d\tau' \\
%& = E_{\sigma}(0) + \int_{0}^{t} \int_{\mathbb{R}} \frac{1}{2} \partial_{t} \left|U_{t}\right|^2 + \frac{1}{2}\partial_t \left|\nabla U\right|^2 - \frac{1}{p+1} \partial_{t} |U|^{p+1}\ dxds \\
%& = E_{\sigma}(0) + \frac{1}{2}\int_{0}^{t} \int_{\mathbb{R}} \overline{U}_{t} U_{tt} + \overline{U}_{xt} U_{x}  - \overline{U}_{t} |U|^{p-1}U \ dx ds + C.C. \\
& = E_{\sigma}(0) 
+ \int_{0}^{\tau} \int_{\R} 
\textrm{Re} \{[- \Delta U + U_{tt}  +   |U|^{p-1} U ] \overline{U}_{t}\}(x,\tau')
\ dx d\tau'\\
& = E_{\sigma}(0) +   \int_{0}^{\tau} \int_{\R} 
\mathrm{Re} \{[ -e^{\sigma|D|}(|u|^{p-1} u) + |U|^{p-1} U ] \overline{U}_{t} \} (x,\tau')\ dx d\tau'\\
& \leq E_{\sigma}(0) 
+ \int_{0}^{\tau} 
\| [|U|^{p-1} U - e^{\sigma|D|}(|u|^{p-1} u)](\cdot,\tau') \|_{L^{2}(\R)} \,
\| U_{t}(\cdot,\tau') \|_{L^{2}(\R)} \ d\tau'.
%\\
%& = E_{\sigma}(0) + \\
%& \quad  + \int_{0}^{t} \left\| e^{\sigma|D|}(|e^{-\sigma |D|}U|^{p-1} e^{-\sigma|D|}U) - |U|^{p-1} U \right\|_{L^{2}_{x}} \| U_{t} (\cdot, s)\|_{L^{2}_{x}} \ ds.
\end{aligned}
\end{equation}

Furthermore, applying Lemma \ref{op} to the last line above, we obtain the estimates
\begin{equation}\label{cased1}
E_{\sigma}(\tau) 
\leq E_{\sigma}(0)
+ C \sigma^{\theta} \int_{0}^{\tau}  \| U(\cdot,\tau') \|_{L^{2}(\mathbb{R})}^{\frac{p+1-2\theta}{2}} E^{\frac{p+1+2\theta}{4}}_{\sigma}(\tau') \ d\tau',
\quad 0 \leq \theta \leq 1,
\end{equation}
for $d = 1$, and
\begin{equation}\label{cased2}
E_{\sigma}(\tau) 
\leq E_{\sigma}(0)
+ C \sigma^{\theta} \int_{0}^{\tau} \| U(\cdot,\tau') \|_{L^{2}(\mathbb{R}^{2})}^{p-1+\theta} E_{\sigma}^{\frac{2-\theta}{2}}(\tau') \ d\tau', \quad 0<\theta<1,
\end{equation}
for $d = 2$, where $C > 0$  is a generic constant.  
Here, we have also used the fact that
\[
\| \nabla U(\cdot,\tau) \|_{L^{2}(\R)} \leq E_{\sigma}^{1/2}(\tau) 
\quad \textrm{and} \quad 
\| U_{t}(\cdot,\tau) \|_{L^{2}(\R)} \leq E_{\sigma}^{1/2}(\tau).
\]

\quad

We first treat the case $d=1$.  Equation \eqref{lowerorder} and the hypothesis H$(t)$ imply that, for $0 \leq \tau' \leq \tau \leq t \leq T$,
\begin{equation}\label{eq:Ut}
\begin{aligned}
\| U(\cdot,\tau') \|_{L^{2}(\mathbb{R})} 
& \leq \| U(\cdot,0) \|_{L^{2}(\mathbb{R})} 
+ \int_{0}^{\tau'} E_{\sigma}^{1/2}(z)\ dz \\
& \leq \| U(\cdot,0) \|_{L^{2}(\mathbb{R})} 
+ (4E_{\sigma}(0))^{1/2}T \\
& \leq \Big(\| U(\cdot,0) \|_{L^{2}(\mathbb{R})} 
+ 2E^{1/2}_{\sigma}(0)\Big)(1+T).
\end{aligned}
\end{equation}
Inserting this into equation \eqref{cased1} and using again H$(t)$ we get
\begin{equation}\label{energyd1}
\begin{aligned}
E_{\sigma}(\tau) 
%& \leq E_{\sigma}(0) 
%+ C \sigma^{\theta} \int_{0}^{\tau}  
%\Big[ \Big(\| U(\cdot,0) \|_{L^{2}(\mathbb{R})} 
%+ 2E^{1/2}_{\sigma}(0)\Big)(1+T) \Big]^{\frac{p-1+2\theta}{2}} 
%E^{\frac{p+3-2\theta}{4}}_{\sigma}(\tau') \ d\tau' \\
& \leq E_{\sigma}(0) 
+ C' \sigma^{\theta} (1+T)^{\frac{p+1-2\theta}{2}}T \\
& \leq E_{\sigma}(0) 
+ C' \sigma^{\theta} (1+T)^{\frac{p+3-2\theta}{2}},
\end{aligned}\end{equation}
where
\[
C' 
:= C \Big(\| U(\cdot,0) \|_{L^{2}(\mathbb{R})} + 2E^{1/2}_{\sigma}(0)\Big)^{\frac{p+1-2\theta}{2}}(4E_{\sigma}(0))^{\frac{p+1+2\theta}{4}}.
\]
It follows that
\[
E_{\sigma}(\tau) \leq 2 E_{\sigma}(0),
\]
provided that
\begin{equation}\label{eq:goal1}
C' \sigma^{\theta} (1+T)^{\frac{p+3-2\theta}{2}} \leq E_{\sigma}(0),
\end{equation}
or, more simply,
\begin{equation}\label{eq:goal2}
\sigma \leq C (1+T)^{-\frac{p+3-2\theta}{2\theta}},
\end{equation}
for some constant $C > 0$.  It is easy to see that the exponent in the expression on the right-hand side is maximized when $\theta = 1$, which yields the desired result.\\

As for $d=2$, we start from \eqref{cased2}
and proceed similarly to \eqref{eq:Ut} and \eqref{energyd1} to get
\[\begin{aligned}
E_{\sigma}(\tau) 
%& \leq E_{\sigma}(0) + C \sigma^{\theta} \int_{0}^{\tau} 
%\| U(\cdot,\tau') \|_{L^{2}(\mathbb{R}^{2})}^{p-2+\theta} E_{\sigma}^{\frac{3-\theta}{2}}(\tau') \ d\tau' \\
& \leq E_{\sigma}(0) + C'' \sigma^{\theta} (1+T)^{p+\theta},
\end{aligned}\]
where
\[
C'' := \left(\| U(\cdot,0) \|_{L^{2}(\mathbb{R}^{2})} + 2 E^{1/2}_{\sigma}(0)\right)^{p-1+\theta} (4E_{\sigma}(0))^{\frac{2-\theta}{2}}.
\]
As in the previous case, the conclusion C$(t)$  follows if
\[
C'' \sigma^{\theta} (1+T)^{p+\theta} \leq E_{\sigma}(0),
\]
which holds if
\[
\sigma \leq C(1+T)^{-\frac{p+\theta}{\theta}}
\]
for some constant $C$.  As before, the exponent in this expression is maximum when $\theta = 1$.  Since we have the restriction $0 < \theta < 1$, we choose $\theta = 1 - \varepsilon$ for any $\varepsilon > 0$.  The desired result follows.  %The desired result then follows from the fact that
%\[
%p + 1 < \frac{p + 1 - \varepsilon}{1 - \varepsilon}
%\]
%for any $\varepsilon > 0$ and $p > 1$.
\end{proof}

\begin{proof}[Proof of $b)$]
Fix $t \in [0,T]$, and suppose that $E_{\sigma}(\tau) \leq 2 E_{\sigma}(0)$ for $0 \leq \tau \leq t$.  If $t = T$, then H$(t')$ holds for all $t'$ in any neighborhood of $t$, and there is nothing to prove.  So assume that $0 \leq t < T$.  Then for any $\delta > 0$ we have
\[
\sup_{\tau \in (t - \delta, t]} E_{\sigma}(\tau) \leq 2 E_{\sigma}(0).
\]
Thus, H$(t')$ holds for $t' \in (t-\delta, t]$.  It remains to check $t' \in [t, t + \delta)$.  
From the definition of $E_{\sigma}$ and from equation \eqref{lowerorder} we have that
\[
\| u(\cdot,t) \|_{G^{\sigma, 1}(\R)} 
+ \| u_{t}(\cdot,t) \|_{G^{\sigma, 0}(\R)} 
< \infty.
\]
We may then apply the local existence theory from Section \ref{local} to construct solutions which exist on an interval $[\tau,\tau+\delta) \subset [0,T]$ for some small $\delta > 0$.  In particular, we can do this so that
\[
\sup_{\tau' \in [\tau,\tau+\delta)} E_{\sigma}(\tau') \leq 4 E_{\sigma}(0).
\]
Thus, H$(t')$ is true for all $t' \in (t - \delta, t + \delta)$.
\end{proof}

\begin{proof}[Proof of $c)$]
Let $\{t_n\}_{n \in \N}$ be a sequence in $[0,T]$ such that $t_n \rightarrow t \in [0,T]$.  Suppose that C$(t_n)$ holds for all $n \in \mathbb{N}$.  Then $E_{\sigma}(t_n) \leq 2 E_{\sigma}(0)$, for every $n \in \mathbb{N}$.  By construction, the $H^{1}$--norm of $u$ and the $L^{2}$--norm of $u_{t}$ are continuous functions in time.  By the Sobolev embedding $H^{1} \hookrightarrow L^{p+1}$, we have that the $L^{p+1}$--norm of $u$ is also continuous in time.  It follows that $E_{\sigma}(t)$ is continuous, so that
\[
E_{\sigma}(t) = \lim_{n \rightarrow \infty} E_{\sigma}(t_n) \leq 2 E_{\sigma}(0).
\]
Consider now $\tau \in [0,t)$.  Since $t_n \rightarrow t$, there exists $n_0 \in \mathbb{N}$ so that $0 \leq \tau \leq t_{n_0}$.  It follows that $E_{\sigma}(\tau) \leq 2 E_{\sigma}(0)$.  Therefore, C$(t)$ holds.
\end{proof}

\begin{proof}[Proof of $d)$]
H$(0)$ is obviously true.
\end{proof}

Based on the above results, we may close the bootstrap, and it follows that C$(t)$ holds for all $t \in [0,T]$.  Thus, we have proven Theorem \ref{mainthm2}. 

\subsection{Case 2: $s \neq 1$}

To conclude the proof for the general case, we observe that by Lemma \ref{embed}, if $(u_0, u_1) \in G^{\sigma_0, s} \times G^{\sigma_0, s-1}$, then $(u_0, u_1) \in G^{\sigma_0/2, 1} \times G^{\sigma_0/2, 0}$.  Thus, we may use the theory from the previous section to show that
\[
u \in C([0,T]; G^{\sigma(T), 1}) \cap C^{1}([0,T]; G^{\sigma(T), 0}).
\]
Applying Lemma \ref{embed} once more, we have that
\[
u \in C([0,T]; G^{\sigma(T)/2, s}) \cap C^{1}([0,T]; G^{\sigma(T)/2, s-1}).
\]
Thus, the desired result holds for general $s$. \\

\noindent \textbf{Acknowledgments}. The authors would like to thank Achenef Tesfahun for his helpful comments.

%%%%%%%%%%%%%%%%%%%%%%%%%%%%%%%%%%%%%%%%%%%%%%%%%%%%%%%%%%%%%%%%
%\bibliographystyle{amsplain}
\bibliographystyle{siam}
\bibliography{database}
%%%%%%%%%%%%%%%%%%%%%%%%%%%%%%%%%%%%%%%%%%%%%%%%%%%%%%%%%%%%%%%%

\end{document}